\documentclass[12pt,a4paper,oneside]{amsart} 

\usepackage{cite}
\usepackage[english]{babel}
\usepackage{geometry}
\usepackage{amsmath}
\usepackage{amssymb}

\title[Monoids of integer-valued polynomials]{Relative polynomial closure
and monadically Krull monoids of integer-valued polynomials}
\author[S.~Frisch]{Sophie Frisch}
\address{Institut f\"ur Mathematik A, Technische Universit\"at Graz,
Steyrergasse 30, 8010 Graz, Austria}
\email{frisch@math.tugraz.at}
\thanks{This publication was supported by the Austrian Science Fund FWF, 
grant P23245-N18}


\long\def\comment#1{\relax}
\long\def\commented_out#1{\relax}

\newtheorem{thm}{Theorem}[section]
\newtheorem{prop}[thm]{Proposition}
\newtheorem{lem}[thm]{Lemma}
\newtheorem{cor}[thm]{Corollary}
\theoremstyle{definition}
\newtheorem{Def}[thm]{Definition}

\theoremstyle{remark}
\newtheorem{rem}[thm]{Remark}

\newcommand{\monoid}[1]{\mathopen{[\![}#1\/\mathclose{]\!]}}
\newcommand{\closure}[2]{C_{#1}(#2)}

\def\C{\mathcal{C}} 
\def\F{\mathcal{F}} 
\def\H{\mathcal{H}} 
\def\P{\mathcal{P}} 

\newcommand{\N}{\mathbb{N}} 
\newcommand{\Z}{\mathbb{Z}} 
\newcommand{\tif}{\tilde f} 

\newcommand{\divides}{\mathrel{\mathop{\mid}}}
\DeclareMathOperator\vP{v_{P}}  
\DeclareMathOperator\dS{d_{S}}  
\DeclareMathOperator\Int{\mathrm{Int}}  
\DeclareMathOperator\Spec{\mathrm{Spec}}  

\begin{document}


\comment{
Let $D$ be a Krull domain, $K$ its quotient field, $S\subseteq D$
and $\Int(S,D)=\{f\in K[x]\mid f(S)\subseteq D\}$ the ring of
integer-valued polynomials on $S$. 
For $f\in \Int(S,D)$, we explicitly construct a divisor homomorphism
from the divisor-closed submonoid generated by $f$ to a finite sum  
of copies of $(\N_0,+)$. This implies that the multiplicative monoid 
$\Int(S,D)\setminus\{0\}$ is monadically Krull. In the case that $D$
is a discrete valuation domain, we give the divisor theories of various
submonoids of $\Int(S,D)$.
In the process, we modify the concept of polynomial closure in such a 
way that every subset of $D$ has a finite polynomially dense subset.
}

\begin{abstract}
Let $D$ be a Krull domain and $\Int(D)$ the ring of integer-valued
polynomials on $D$. For any $f\in \Int(D)$, we explicitly construct a 
divisor homomorphism from $\monoid{f}$, the divisor-closed submonoid of 
$\Int(D)$ generated by $f$, to a finite sum  of copies of $(\N_0,+)$. 
This implies that $\monoid{f}$ is a Krull monoid. 

For $V$ a discrete valuation domain, we give explicit divisor theories 
of various submonoids of $\Int(V)$. In the process, we modify the concept 
of polynomial closure in such a way that every subset of $D$ has a finite 
polynomially dense subset. The results generalize to $\Int(S,V)$, the
ring of integer-valued polynomials on a subset, provided $S$ doesn't
have isolated points in $v$-adic topology.
\end{abstract}

\maketitle

\section{Introduction}

The ring of integer-valued polynomials $\Int(\Z)$ enjoys quite
chaotic non-unique factorization, in the following sense: given 
any finite list of natural numbers $1<n_1\le n_2\le\ldots\le n_k$, 
one can find a polynomial $f\in\Int(\Z)$ that has exactly $k$ 
essentially different factorizations into irreducible elements of 
$\Int(\Z)$, namely, one with $n_1$ irreducible factors, one with 
$n_2$, etc.\ \cite{Fri13len}.
In contrast to this, A.~Reinhart \cite{Rei14monK} has shown 
that $\Int(D)$ is monadically Krull for any unique factorization 
domain $D$, which means that the divisor-closed submonoid $\monoid{f}$ 
generated by any single polynomial $f\in\Int(D)$ is a Krull monoid.
So we have here an interesting case of Krull monoids with rather
wild factorization properties.

In this paper we examine the ring of integer-valued polynomials 
on a subset of a Krull domain, and the divisor closed submonoid 
$\monoid{f}$ generated by a single polynomial $f\in\Int(S,D)$. 
If $S$ doesn't have any isolated points in any of the topologies
given by essential valuations of $D$, we can
construct a divisor homomorphism from $\monoid{f}$ to a finite direct 
sum of copies of $(\N_0,+)$ [Theorem \ref{Krull_divisor_hom}].
This implies that $\monoid{f}$ is a Krull monoid, and hence, that 
$\Int(S,D)$ is monadically Krull. 

In the special case where $D$ is a discrete valuation domain,
we can actually determine the divisor theories of certain
submonoids of $\Int(S,D)$ 
[Proposition \ref{local_divisor_hom} and 
Theorem \ref{monadic_local_divisor_theory}].

For the purpose of constructing divisor homomorphisms on monoids of 
integer-valued polynomials, we will study ``relative'' polynomial 
closure, that is, polynomial closure with respect to a subset of $K[x]$, 
in section~2. 
This modification 
of the concept of polynomial closure makes it possible to find finite 
polynomially dense subsets of arbitrary sets in section~3. 
Equipped with these finite polynomially dense sets we construct the 
actual divisor homomorphisms and, in some cases, divisor theories,
to finite sums of copies of $(\N_0,+)$ in sections~4 and 5.

A short review of integer-valued polynomial terminology: Let $D$ be 
a domain with quotient field $K$, $S\subseteq K$ and $f\in K[x]$. 
$f$ is called integer-valued if 
$f(D)\subseteq D$ and $f$ is called integer-valued on $S$ if 
$f(S)\subseteq D$. If there are several possibilities for $D$, we 
say $D$-valued on $S$ instead of integer-valued on $S$. 

The ring of integer-valued polynomials on $D$ is written $\Int(D)$, and 
the ring of integer-valued polynomials on a subset $S$ of the quotient 
field of $D$ is denoted by $\Int(S,D)$:
\[
\Int(S,D)=\{f\in K[x]\mid f(S)\subseteq D\},\quad \Int(D)=\Int(D,D).
\]

\begin{Def}\label{f_monoid_def}
Let $D$ be a domain with quotient field $K$, $S\subseteq D$ and 
$f\in\Int(S,D)$. The divisor-closed submonoid of $\Int(S,D)$ 
generated by $f$, which we write $\monoid{f}$, is the 
multiplicative monoid consisting of 
all $g\in \Int(S,D)$ for which there exists $m\in \N$ and 
$h\in\Int(S,D)$, such that $g\cdot h= f^m$. 
\end{Def}

Keep in mind that the elements of $\monoid{f}$ are not just
polynomials in $\Int(S,D)$ that divide some power of $f$ in
$K[x]$. The co-factor is also required to be in $\Int(S,D)$.
We will frequently use the following divisibility criterion for
$\monoid{f}$.

\begin{rem}\label{divisibility}
Let $\monoid{f}$ be the divisor closed submonoid of $\Int(S,D)$ as
in Definition $\ref{f_monoid_def}$ and $a,b\in\monoid{f}$.
Then $a$ divides $b$ in $\monoid{f}$ if and only if $a$ divides $b$
in $K[x]$ and the cofactor $c=b/a$ is in $\Int(S,D)$.
\end{rem}

Multiplying a polynomial in $\monoid{f}$ by a constant in $D$ 
does not in general result in an element of $\monoid{f}$. We can 
multiply elements of $\monoid{f}$ by some suitable constants, though.

\begin{lem}\label{mult_by_constant}
Let $V$ be the valuation domain of a valuation $v$ on $K$, $S\subseteq V$,
$f\in\Int(S,V)$ and $\monoid{f}$ the divisor-closed submonoid of
$\Int(S,V)$ generated by $f$. Let $g\in\monoid{f}$ and $a\in K$.
If $-\min_{s\in S}v(g(s))\le v(a)\le 0$ then $ag\in\monoid{f}$.
\end{lem}

\begin{proof}
Let $g,h\in\Int(S,V)$ and $m\in\N$ such that $gh=f^m$. Then
both $ag$ and $a^{-1}h$ are in $\Int(S,V)$, and $ag\cdot a^{-1}h=f^m$.
\end{proof}

We recall the definitions of ideal content and fixed divisor,
whose interplay will be an important ingredient of proofs.
Let $R$ be a domain and $f\in R[x]$. The content of $f$,
denoted $c(f)$, is the fractional ideal generated by the 
coefficients of $f$. 
If $R$ is a principal ideal domain, we identify, by abuse of notation,
ideals by their generators and say that $c(f)$ is the gcd of the 
coefficients of $f$. A polynomial $f\in R[x]$ is called primitive if 
$c(f)=R$, that is, in the case of a PID, if $c(f)=1$.

\begin{Def} \label{fixed_divisor_def}
Let $D$ be a domain with quotient field $K$,
$S\subseteq D$ and $f\in K[x]\setminus\{0\}$.
The \textit{fixed divisor of $f$ on $S$}, denoted $\dS(f)$, is the 
$D$-submodule of $K$ generated by the image $f(S)$. Note that
$\dS(f)$ is a fractional ideal. If $S=D$, we write $\mathrm{d}(f)$ for 
$\mathrm{d_D}(f)$.
If $D$ is a PID, we will, by abuse of
notation, sometimes write a generator to stand for the ideal, e.g.,
$\dS(f)=1$ for $\dS(f)=D$. A polynomial $f\in\Int(S,D)$ is called
\textit{image-primitive} if $\dS(f)=D$.
\end{Def}

For polynomials in $D[x]$, image-primitive implies primitive, but
not vice versa.
One difference between ideal content and fixed divisor is that
the ideal content is multiplicative for sufficiently nice rings (called
Gaussian rings), including principal ideal rings, whereas the fixed
divisor is not multiplicative. $\dS(f)\dS(g)$ contains $\dS(fg)$, but
the containment can be strict. 

\begin{rem} \label{imageprimitive_rem}
Two easy but useful facts:
\begin{enumerate}
\item
If $f\in\Int(S,D)$ is image-primitive then $f^n$ is image-primitive
for all $n\in\N$.
\item
If $f\in\Int(S,D)$ is image-primitive then all divisors in $\Int(S,D)$
of $f$ are also image-primitive.
\end{enumerate}
\end{rem}

\begin{rem} \label{min_v_notation}
In case $D$ is an intersection of valuation rings, then every
$f\in\Int(S,D)$ is also in $\Int(S,V)$ for all these valuation rings,
and $f$ may be image-primitive as an element of $\Int(S,V)$, but
not as an element of $\Int(S,D)$. In this case, we write 
\[
v(f(S)):= \min_{s\in S} v(f(s))
\]
and write $v(f(S))=0$ to express that $f$ is image-primitive when
regarded as an element of $\Int(S,V)$.
\end{rem}
 
Regarding valuation terminology: we use additive valuations, that 
is, a valuation is a map 
$v\colon K\setminus\{0\}\rightarrow \Gamma$, where $(\Gamma,+)$ is
a totally ordered group, satisfying
\begin{enumerate}
\item 
$v(ab)=v(a)+v(b)$
\item
$v(a+b)\ge \min(v(a),v(b))$
\end{enumerate}
and we set $v(0)=\infty$. The valuation ring of a valuation $v$
on a field $K$ is $V=\{k\in K\mid v(k)\ge 0\}$ and the valuation
group is the image of $v$ in $\Gamma$.

\section{Relative polynomial closure}

\begin{Def}[relative polynomial closure]\label{rel_poly_colsure}
Fix a domain $D$ with quotient field $K$. Let $T\subseteq K$ and
$\F\subseteq K[x]$. 

The polynomial closure of $T$ relative to $\F$ is 
\[
\closure{\F}{T}=
\{s\in K\mid  \forall f\in \F\cap\Int(T,D):\kern5pt f(s)\in D \}.
\]

If $T\subseteq S\subseteq K$, and $\closure{\F}{T}\supseteq S$ we call $T$ 
polynomially dense in $S$ relative to $\F$.

The definition of polynomial closure and polynomial density depends on
the choice of $D$. If there is any doubt about $D$, we say $D$-polynomial 
closure and $D$-polynomially dense.
\end{Def}

Polynomial closure relative to $K[x]$ is the ``usual'' polynomial closure,
introduced by Gilmer \cite{Gil89StDiv} and studied by 
McQuillan \cite{McQ91Gthm}, the present author \cite{Fri96scp},
Cahen \cite{Cah96pc}, Park and Tartarone \cite{PaTa05pced}
and Chabert \cite{Cha10PCVF}, among others.
The reason why we generalize the well-known concept of polynomial closure 
will become apparent in the next section: when we consider polynomial 
closure relative to a set of polynomials whose irreducible factors are 
restricted to a finite set, it becomes possible to find finite 
polynomially dense subsets of any fractional set.

\begin{rem}\label{closure_properties}
The following properties of polynomial closure relative to a subset
$\F$ of $K[x]$ are easy to check.
\begin{enumerate}
\item
$\closure{\F}{T}=\bigcap_{f\in \F\cap\, \Int(T,D)} f^{-1}(D)$
\item
Polynomial closure relative to $\F$ is a closure operator,
in the sense that
\begin{enumerate}
\item 
$T\subseteq \closure{\F}{T}$
\item
$\closure{\F}{\closure{\F}{T}}=\closure{\F}{T}$
\item
$T\subseteq S\Longrightarrow \closure{\F}{T}\subseteq \closure{\F}{S}$
\end{enumerate}
\item
Polynomial closure relative to $\F$ is the closure given by a Galois 
correspondence that maps every subset $T$ of $K$ to a subset of $\F$, 
and every subset $G$ of $\F$ to a subset of $K$, namely, 
\[
T\mapsto \F\cap \Int(T,D)  \quad\text{and}\quad
G\mapsto \bigcap_{f\in G} f^{-1}(D).  
\]
\item
If $\F_0\subseteq \F_1\subseteq K[x]$ then 
$\closure{\F_1}{T}\subseteq \closure{\F_0}{T}$.
\item
If $T$ is polynomially dense in $S$ relative to $\F_1$, and
$\F_0\subseteq \F_1$, then $T$ is polynomially dense in $S$ 
relative to $\F_0$.
\end{enumerate}
\end{rem}

When the domain $D$ is a valuation ring, then polynomially dense subsets 
of $S$ relative to $\F$ are easily characterized:

\begin{lem}\label{DensityCriterion} 
Let $v$ be a valuation on a field $K$, $V$ its valuation ring,
$T\subseteq S\subseteq K$ and $\F\subseteq K[x]$. Consider
\begin{enumerate}
\item
$\forall f\in\F$\ 
$\min_{t\in T}v(f(t))=\min_{s\in S} v(f(s))$
\item
$T$ is $V$-polynomially dense in $S$ relative to $\F$.
\end{enumerate}
{\rm (1)} implies {\rm (2)}. If $\F$ is closed under multiplication
by non-zero constants in $K$ then {\rm (2)} implies {\rm (1)}.
\end{lem}

\begin{proof}
$(1 \Rightarrow 2)$ For every polynomial $f\in\F\cap\Int(T,V)$,
$\min_{t\in T} v(f(t))\ge 0$. Therefore, by $(1)$,
$\min_{s\in S} v(f(s))\ge 0$ and hence $f\in\Int(S,V)$.

$(2 \Rightarrow 1)$ For every $f\in\F$,
$\min_{t\in T} v(f(t))\ge \min_{s\in S} v(f(s))$, since $T\subseteq S$.
If $f\in\F$ and $\alpha\in\Z$ are such that 
$\min_{t\in T} v(f(t))\ge \alpha>\min_{s\in S} v(f(s))$,
pick $a\in K$ with $v(a)=-\alpha$. 
Then $af\in\F\cap\Int(T,V)$, but $af\not\in\Int(S,V)$, so $T$ is
not $V$-polynomially dense in $S$ relative to $\F$.
\end{proof}

\section{Finite polynomially dense subsets}

Let $F$ be a finite set of irreducible polynomials in $K[x]$
and $\F$ the multiplicative submonoid of $K[x]$ generated
by $F$ and the non-zero constants of $K$. That is, $\F$
consists of all non-zero polynomials in $K[x]$ whose irreducible
factors in $K[x]$ are (up to multiplication by non-zero constants) 
in $F$.

We will now construct, for every subset $S$ of a discrete
valuation ring $V$, a finite polynomially dense subset of 
$S$ relative to $\F$.
It is possible to admit fractional subsets of $K$, but for
simplicity's sake we restrict ourselves to subsets of $V$.

By discrete valuation, we mean, more precisely, a discrete rank $1$ 
valuation, that is, a valuation $v$ whose value group 
is isomorphic to $\Z$. A normalized discrete valuation is one whose 
value group is actually equal to $\Z$. The valuation ring of a discrete 
valuation is called discrete valuation ring, abbreviated DVR.
As we all know, a DVR is a local principal ideal domain. 

\begin{rem}\label{zlemma}
Let $v$ be a discrete valuation on $K$ with valuation ring $V$,
$f\in K[x]$, and $L\supseteq K$ a finite-dimensional field 
extension over which $f$ splits. 
Let $w$ be an extension of $v$ to $L$ ($w\mid_K=v$),
$W$ the valuation ring of $w$ and $P$ its maximal ideal.
Say $f$ splits as 
$
f(x)=c\prod_{j=1}^k(x-b_j)\prod_{j=1}^m(x-a_j)
$
with $w(b_j)< 0$ for $1\le j\le k$ and $w(a_j)\ge 0$ for $1\le j\le m$
over $L$. 

Then for all $s\in V$,
\[
v(f(s))=w(c)+\sum_{j=1}^k w(b_j) + \sum_{j=1}^m w(s-a_j)
\]
\end{rem}

\begin{proof} This follows from the fact that $w(s\pm b)= w(b)$
whenever $w(b)<w(s)$.
\end{proof}

\begin{Def}\label{isol_point__def}
Let $X$ be a topological space and $S\subseteq X$. An isolated point
of $S$ is an element $s\in X$ having a neighborhood $U$ such that
$U\cap S=\{s\}$.
\end{Def}

\begin{prop} \label{finite_dense_set}
Let $v$ be a discrete valuation on $K$ and $V$ its valuation ring.
Let $F\ne\emptyset$ be a finite set of 
monic irreducible polynomials in $K[x]$ and $\F$ the set of those 
polynomials in $K[x]$ whose monic irreducible factors are all in $F$.
Let $S\subseteq V$.

\begin{enumerate}
\item
Then there exists a finite subset $T\subseteq S$ such that
\[
\forall f\in\F\; \min_{t\in T}v(f(t))=\min_{s\in S}(v(f(s)))
\]
and every such $T\subseteq S$ is, in particular, a finite set
that is polynomially dense in $S$ relative to $\F$.
\item
If no root of any $f\in F$ is an isolated point of $S$ in $v$-adic
topology, then the above set $T$ can be chosen such as not to 
contain any root of any $f\in F$.
\item
Let $L$ be the splitting field of $F$ over $K$, $w$ an extension of
$v$ to $L$ and $W$ the valuation ring of $w$. Let $A$ be the set of
distinct roots of polynomials of $F$ in $W$. Then $T$ in (1) and (2)
can be chosen with $\left|T\right|\le \max(1, \left|A\right|)$.
\end{enumerate}
\end{prop}

\begin{proof}
Let $L$, $w$, $W$, and $A$ as in (3). Let $P$ be the maximal ideal of $W$.
We call the elements of $A$ ``the roots''. We may assume $S\ne\emptyset$ 
and $A\ne\emptyset$ (otherwise the claimed facts are trivial).
In view of Remark \ref{zlemma} it suffices to construct a set 
$T\subseteq S$ such that, for every finite sequence $(a_i)_{i=1}^m$ in $A$, 
\[\min_{t\in T}\sum_{i=1}^m w(t-a_i)=\min_{s\in S}\sum_{i=1}^m w(s-a_i)\]
We will do this by constructing a finite covering $\C$ of $S$ by disjoint 
sets $C\subseteq W$ and for each $C\in\C$ choosing a representative 
$t\in C\cap S$ such that $w(t-a)\le w(s-a)$ for every $a\in A$ and every 
$s\in C\cap S$. 
This representative $t\in C\cap S$ then satisfies $\forall f\in\F$ 
$v(f(t))=\min_{s\in C\cap S}v(f(s))$, by Remark \ref{zlemma}. If we 
take $T$ to be the set of representatives of covering sets $C\in\C$ 
then for every $f\in\F$, $\min_{s\in S}v(f(s))$ is realized by 
some $s\in T$. By Lemma \ref{DensityCriterion}, this makes $T$
polynomially dense in $S$ relative to $\F$.

For any ideal $I$ of $W$, we call a residue class $r+I$ ``relevant''
if $S\cap (r+I)\ne\emptyset$. 

We construct $\C$, $\C_n$ ($n\ge 0$) and $T$ inductively. Before step $0$, 
initialize $T=\emptyset$, $\C=\emptyset$, $\C_0=\{W\}$. 

At the beginning of step $n$, $\C$ is a finite set of relevant 
residue classes of various $P^k$ with $k<n$ while $\C_n$ is a finite 
set of relevant residue classes of $P^n$ each containing at least one root. 
In step $n$, initialize $\C_{n+1}=\emptyset$; then go through each 
$C\in \C_n$ and process it as follows:
\begin{enumerate}
\item
If $C\cap S=\{c\}$ with $c\in A$ then put $c$ in $T$ and $C$ in $\C$.
Note that in this case $C\cap V$ is a $v$-adic neighborhood of $c$
whose intersection with $S$ is $\{c\}$, and that therefore $c\in A$ is 
an isolated point of $S$.
\item
Else, if $C$ contains a relevant residue class $D$ of $P^{n+1}$ 
which doesn't contain a root, pick such a $D$, add a representative 
of $D\cap S$ to $T$; then put $C$ in $\C$.
\item
Else place all relevant residue classes of $P^{n+1}$ contained in $C$
(each containing a root, by construction) in $\C_{n+1}$.
\end{enumerate}
If $\C_{n+1}$ is empty at the end of step $n$, stop. Otherwise
proceed to step $n+1$.

Note that after each step $n$, $\C\cup\C_{n+1}$ is a covering of $S$.
When the algorithm terminates with $\C_{n+1}=\emptyset$, then $\C$ is
a covering of $S$ and $T$ contains for each $C\in\C$ a representative 
$t\in C\cap S$ satisfying $w(t-a)=\min_{s\in C\cap S}w(s-a)$ for all $a\in A$.
Therefore $v(f(t))=\min_{s\in C\cap S}v(f(s))$ for all $f\in\F$ by Remark
\ref{zlemma}.

The algorithm terminates when no root is left in $\bigcup\C_{n+1}$.
For each root $a\in A$, one can give an upper bound on $n$ such 
that $a$ is no longer in $\C_{n+1}$. Namely, let $n$ such that 
$w(a-a')<n$ for all roots $a\ne a'$. If $(a+P^{n+1})\cap S=\emptyset$
then a residue class containing $a$ has been dropped as not relevant
at or before step $n$, so $a+P^{n+1}\not\in\C_{n+1}$. 
If $(a+P^{n+1})\cap S=\{a\}$, then a residue class containing $a$ is
placed in $\C$ at step $n+1$ or earlier. Otherwise, $a+P^{n+1}$ contains
an element of $S$ other than $a$. Let $s\in (a+P^{n+1})\cap S$, with 
$w(s-a)=m$ minimal. Then $a+P^m$ will be placed in $\C$ by step $m$.

This shows (1). For (2), note that the set $T$ thus constructed 
contains no root of any $f\in F$ except such as are isolated points 
of $S$ in $v$-adic topology. For (3), note that every time an
element is added to $T$, a set containing at least one root
is transferred from $\C_n$ to $\C$ and the number of roots in 
$\bigcup_{C\in\C_n}C$ decreases.
\end{proof}

\begin{rem}
Thanks to the anonymous referee for pointing out that parts (1) and (2) 
of Proposition \ref{finite_dense_set} can be show more quickly by
applying Dickson's theorem \cite{GeHK06nuf}[Thm.~1.5.3],
which says that the set of
minimal elements of any subset $N$ of $\N_0^{m}$ is finite and 
that for every $a\in N$ there exists a minimal element $b\in N$ with 
$b\le a$, to the subset $N=\{ (w(s-a))_{a\in A}\mid s\in S\}$ of
$\N_0^A$.
\end{rem}

%
%

\section{Divisor theories for monoids of integer-valued polynomials
on discrete valuation rings}

A short review of monoid terminology used in the definition of 
divisor homomorphism:
By \textit{monoid} we mean a semigroup that has a neutral element.
All monoids we consider here are commutative, and they are cancellative,
that is, whenever $ab=cb$ or $ba=bc$, it follows that $a=c$.

Let $(M,+)$ be a commutative monoid, written additively, and $a,b\in M$.
\begin{enumerate}
\item
We say that $a$ \textit{divides} $b$ in $M$, and write $a\mid b$, whenever 
there exists $c\in M$ such that $a+c=b$. 
\item
We call an element $d\in M$ a \textit{greatest common divisor},
abbreviated gcd, of a subset $A\subseteq M$, if
\begin{enumerate} 
\item
$d\mid a$ for all $a\in A$
\item
for all $c\in M$: if $c\mid a$ for all $a\in A$ then $c\mid d$.
\end{enumerate}
\end{enumerate}

\begin{Def}\label{divisor_hom_def}
A monoid homomorphism $\varphi\colon G\rightarrow H$ is called
a divisor homomorphism if $\varphi(a)\mid \varphi(b)$ in $H$
implies $a\mid b$ in $G$. 

A divisor homomorphism 
$\varphi\colon G\rightarrow \sum_{i=1}^n (\N_0,+)$
is called a divisor theory if each of the basis
vectors $e_i$ (having $1$ in the $i$-th coordinate and zeros
elsewhere) occurs as gcd of a finite set of images $\varphi(g)$.
\end{Def}

We are preparing to construct divisor homomorphisms from certain
submonoids of $\Int(S,D)$, where $D$ is a Krull domain, to finite
sums of copies of $(\N_0,+)$, relating divisibility in $\Int(S,D)$
to divisibility in a finitely generated free commutative monoid,
which a priori looks much simpler. 
If $(M,+)$ is a direct sum of $k$ copies of $(\N_0,+)$, then the
divisibility relation in $M$ is just the partial order given by
the order relations on each component:
Let $a,b\in M$ with $a=(a_1,\ldots, a_k)$ and $b=(b_1,\ldots, b_k)$.
Then $a\mid b$ in $M$ is equivalent to $a_i\le b_i$ for all $1\le i\le k$.
Therefore, any set $\{(m_{i1},m_{i2},\ldots,m_{ik})\mid i\in I\}$ of 
elements of $M$ has a unique gcd, namely, 
$d=(\min_i(m_{i1}), \min_i(m_{i2}),\ldots, \min_i(m_{ik}))$.

In what follows, we denote the normalized discrete valuation on $K(x)$ 
corresponding to an irreducible polynomial $h\in K[x]$ by $v_h$; for 
$g\in K[x]$, $v_h(g)$ is the exponent to which $h$ occurs in the 
essentially unique factorization of $g$ in $K[x]$ into irreducible 
polynomials, and for $g_1/g_2\in K(x)$, $v_h(g_1/g_2)=v_h(g_1)-v_h(g_2)$.

In this section we examine the special case $\Int(S,V)$, where $V$ 
is a discrete valuation ring (DVR). 

\begin{prop} \label{local_divisor_hom}
Let $v$ be a normalized discrete valuation on $K$ and $V$ its valuation ring.
Let $H$ be a finite set of pairwise non-associated irreducible polynomials
in $K[x]$ and $\H$ the multiplicative submonoid of $K[x]$ generated by 
$H$ and the non-zero constants in $K$. 
Let $S\subseteq V$ such that no root of any $h\in H$ is an isolated
point of $S$ in $v$-adic topology.
Let $\F=\H\cap\Int(S,V)$.

There exists a finite subset $T$ of $S$ that is polynomially dense in $S$
relative to $\H$ and contains no root of any $h\in H$;
and for every such $T$
\[
\varphi\colon \F\rightarrow 
\sum_{h\in H}(\N_0,+)\oplus \sum_{t\in T}(\N_0,+),\quad
\varphi(g) = \left((v_h(g)\mid h\in H), (v(g(t))\mid t\in T)\right),
\]
is a divisor homomorphism. If $T$ is chosen minimal, $\varphi$ is a 
divisor theory.
\end{prop}

\begin{proof}
The existence of a finite polynomially dense subset $T$ containing no
root of any $h\in H$ is Proposition \ref{finite_dense_set}. 
Once we have a finite dense set,
a minimal dense set can be obtained by removing redundant elements.

$\varphi$ is clearly a monoid homomorphism.
Now suppose
$a,b\in\F$ such that $\varphi(a)\mid \varphi(b)$, and 
set $c=b/a$. We must show $c\in\Int(S,V)$.

$\varphi(a)\mid \varphi(b)$ means $v_h(a)\le v_h(b)$ for all $h\in H$ 
and $v(a(t))\le v(b(t))$ for all $t\in T$. 
The first shows $c\in K[x]$, and therefore $c\in \H$, and the second 
shows that $c(t)\in V$ for all $t\in T$. 
Since $T$ is polynomially dense in $S$ relative to $\H$, it follows
that $c\in\Int(S,V)$. We have shown $\varphi$ to be a divisor
homomorphism.

It remains to show that every $e_h$ for any $h\in H$ and every $e_t$ 
for any $t\in T$ occurs as the gcd of a finite set of images of 
elements of $\F$, provided $T$ is minimal.

We may assume, without changing $\H$, $\F$ or $\varphi$ in any way, 
that the elements of $H$ are in $V[x]$ and primitive.

First, let $p$ be a generator of the maximal ideal of $V$.
The constant polynomial $p$ is an element of $\F$
satisfying $v_h(p)=0$ for all $h\in H$ and $v(p(t))=1$ for all $t\in T$.

Second, we note that every polynomial $h\in H$ is an element of
$\F$ satisfying $v_h(h)=1$ and $v_l(h)=0$ for every 
$l\in H\setminus\{h\}$.

Third, we show that for every $t\in T$, there exists 
$g_t\in\F$ such that $v(g_t(t))=0$ and $v(g_t(r))>0$ for all
$r\in T\setminus\{t\}$.
We use the minimality of $T$ and Lemma \ref{DensityCriterion}:
Since $T$ is polynomially dense in $S$ relative to $\H$, but
$T\setminus\{t\}$ is not, there exists a polynomial $k\in\H$ with
$v(k(t))=\min_{s\in S} v(k(s))$ and $v(k(r))>\min_{s\in S} v(k(s))$
for all $r\in T\setminus\{t\}$. 
Let $k$ be such a polynomial and $\alpha=v(k(t))$. Then
$g_t(x)=p^{-\alpha}k(x)$ has the desired properties.

Fourth, we show that for every $t\in T$ and $h\in H$ there exists
$g_{t h}\in\F$ such that $v(g_{t h}(t))=0$ and $v_h(g_{t h})>0$.
Let $k$ be any polynomial in $\F$ with $v_h(k)>0$. 
If $v(k(t))=\alpha>0$, set
$g_{t h}(x)=p^{-\alpha}k(x)g_t(x)^{\alpha}.$


Now for any $h\in H$ and $t\in T$,
\[
e_h=\gcd(\{\varphi(g_{t h})\mid t\in T\}\cup\{\varphi(h)\})
\quad\text{and}\quad
e_t=\gcd(\{\varphi(g_r)\mid r\ne t\}\cup \{\varphi(p)\}).
\]
\end{proof}

\section{Divisor homomorphisms on monadic monoids of integer-valued
polynomials}

What we have found out about the submonoid of $\Int(S,D)$ consisting of 
polynomials whose irreducible factors in $K[x]$ come from a fixed finite 
set, we now apply to the divisor closed submonoid of $\Int(S,D)$
generated by a single polynomial.

Recall that $\monoid{f}$, the divisor-closed submonoid of $\Int(S,D)$ 
generated by $f$, is the multiplicative monoid consisting of all those 
$g\in \Int(S,D)$ which divide some power of $f$ in $\Int(S,D)$. 
Also, it will be useful to recall the definition of image-primitive,
and of $\dS(f)$, the fixed divisor of $f$ on $S$ from
Definition \ref{fixed_divisor_def}.

First let us get a trivial case out of the way:

\begin{lem}\label{imageprimitive_f}
Let $V$ be a DVR, $S\subseteq V$ and $f\in V[x]$ with $\dS(f)=V$.
Let $F\subseteq V[x]$ be a set of primitive polynomials in $V[x]$
representing the different irreducible factors of $f$ in $K[x]$.
Let $\F_0$ be the multiplicative submonoid of $V[x]$ generated by $F$
and the units of $V$.
Then
\begin{enumerate}
\item
$\monoid{f}=\F_0$
\item
Every element $g$ of $\monoid{f}$ is in $V[x]$, is primitive, and 
satisfies $\dS(g)=V$.
\item
If $g,h\in \monoid{f}$, then $g$ divides $h$ in $\monoid{f}$ if and
only if $g$ divides $h$ in $K[x]$.
\item
$
\varphi\colon \monoid{f}\rightarrow
\sum_{h\in F} (\N_0,+),\quad
\varphi(g) = \left(v_h(g)\mid h\in F\right),
$
is a divisor theory.
\end{enumerate}
\end{lem}

\begin{proof}
We will show (1) and (2). The remaining statements follow from (1).

$f\in V[x]$ is image-primitive on $S$ and hence primitive.
The same holds for all powers of $f$ and for all divisors in $V[x]$
of any power of $f$ by Remark \ref{imageprimitive_rem}.
Therefore every divisor in $V[x]$ of any power of $f$ is in $\F_0$,
and vice versa, every element of $\F_0$ is a divisor in $V[x]$ of some
power of $f$. Therefore every element of $\F_0$ is image-primitive on $S$,
and also $\F_0\subseteq \monoid{f}$.

Now let $g\in\monoid{f}$. Let $m\in\N$ and $h\in\Int(S,V)$ with
$hg=f^m$.
Then $h=c\tilde h$ and $g=d\tilde g$ with $\tilde g,\tilde h\in\F_0$
and $c,d\in K$.
Since $\tilde g$ and $\tilde h$ are image-primitive on $S$, we must have
$v(c)\ge 0$ and $v(d)\ge 0$.
Since $f^m$ is primitive, $v(c) = -v(d)$. It follows that $v(c)=v(d)=0$
and therefore $g,h\in\F_0$.
\end{proof}

Let $D$ be a domain with quotient field $K$, $S$ a subset of $D$,
and $f\in\Int(S,D)$. Let $H$ be a set of representatives 
(up to multiplication by a non-zero constant) of the irreducible factors 
of $f$ in $K[x]$.
For instance, $H$ could be the set of monic irreducible factors of $f$
in $K[x]$. Or, in case that $D$ is a principal ideal domain, such as,
for instance, a discrete valuation domain, $H$ can be chosen to be the
set of primitive irreducible polynomials in $D[x]$ dividing $f$ in $K[x]$.
By $\H$ we denote the multiplicative submonoid of $K[x]\setminus\{0\}$ 
generated by $H$ and the constants in $K\setminus\{0\}$. (Note that $\H$ 
depends only on $f$, not on the choice of $H$). 
Obviously $\monoid{f}\subseteq \H\cap\Int(S,D)$. We now examine when
the equality holds. In this non-trivial case we can give a divisor
theory of $\monoid{f}$ [Theorem \ref{monadic_local_divisor_theory}].
Otherwise, we have to be content with a divisor homomorphism
[Theorem \ref{Krull_divisor_hom}].

\begin{thm}\label{monoids_equal}
Let $V$ be a discrete valuation domain with quotient field $K$,
$S\subseteq V$ and $f\in\Int(S,V)$. Let $\H$ be
multiplicative submonoid of $K[x]$ generated by the irreducible factors
of $f$ and the non-zero constants.
If $\dS(f)\ne V$ then \[\monoid{f}=\H\cap\Int(S,V)\].
\end{thm}

\begin{proof}
Let $H$ be the set of primitive irreducible polynomials in $V[x]$
that divide $f$ in $K[x]$.
Let $f=c(f)\tif$ with $c(f)\in K\setminus\{0\}$ the content of $f$ 
and $\tif\in V[x]$ primitive. For arbitrary $b\in V\setminus\{0\}$,
we show that $b\tif \in\monoid{f}$. 

Let $b\in V\setminus\{0\}$.
Since $\dS(f)\ne V$, $v(\dS(f))>0$ and we may apply the Archimedean
axiom.  Let $m\in\N$ such that $m v(\dS(f))\ge v(b) - v(c(f))$.

Then $f^{m+1} = (f^m c(f) b^{-1}) b\tif$, and both
$(f^m c(f) b^{-1})$ and $b\tif$ are in $\Int(S,V)$.
Therefore $ b\tif\in\monoid{f}$.

Now that $b\tif\in\Int(S,V)$ for arbitrary $b\in V\setminus\{0\}$, all 
factors of $b\tif$ in $V[x]$ are in $\monoid{f}$.
Therefore, all primitive irreducible factors of $f$ 
and all non-zero constants of $V$, and furthermore, all products of such 
elements, are in $\monoid{f}$.  Finally, by Lemma \ref{mult_by_constant}, 
we can multiply elements of $\monoid{f}$ by any constant $a\in K$ with
$v(a)<0$, as long as the result is integer-valued on $S$. Therefore,
$\H\cap\Int(S,V)\subseteq \monoid{f}$. 

The reverse inclusion $\monoid{f}\subseteq \H\cap\Int(S,V)$ is trivial.
\end{proof}

\begin{thm}\label{monadic_local_divisor_theory}
Let $v$ be a normalized discrete valuation on $K$ and $V$ its valuation
ring.
Let $S\subseteq V$ and $f\in\Int(S,V)$, such that no root of $f$ is
an isolated point of $S$ in $v$-adic topology. 
Let $H$ be the set of different monic irreducible factors of $f$ in $K[x]$
and $\H$ the multiplicative submonoid of $K[x]$ generated by $H$ and the
non-zero constants in $K$.
By $\monoid{f}$ denote the divisor-closed submonoid of $\Int(S,V)$ 
generated by $f$. 

There exists a finite polynomially dense subset $T$ of $S$ relative 
to $\H$ that does not contain any root of $f$; and for every such $T$
\[
\varphi\colon \monoid{f}\rightarrow 
\sum_{h\in H}(\N_0,+)\oplus \sum_{t\in T}(\N_0,+)\quad
\varphi(g) = \left((v_h(g)\mid h\in H), (v(g(t))\mid t\in T)\right),
\]
is a divisor homomorphism. If $\dS(f)\ne V$ and $T$ is chosen minimal
then $\varphi$ is a divisor theory.
\end{thm}

\begin{proof}
$\monoid{f}$ is a submonoid of $\F=\H\cap\Int(S,V)$. The
monoid homomorphism $\varphi$ in the theorem is the restriction
of the divisor homomorphism of Proposition \ref{local_divisor_hom}
to $\F$ and therefore itself a divisor homomorphism. If $\dS(f)\ne V$
then $\monoid{f}=\F$ by Theorem \ref{monoids_equal}. In this
case, $\varphi$ is a divisor theory by Proposition \ref{local_divisor_hom},
provided $T$ is minimal.
\end{proof}

Recall that a Krull domain $R$ is a domain satisfying the following
conditions with respect to $\Spec^1(R)$, the set of prime ideals
of height $1$:
\begin{enumerate}
\item
For every $P\in \Spec^1(R)$, the localization $R_P$ is a DVR.
\item
$R=\bigcap_{P\in \Spec^1(R)} R_P$
\item
Each non-zero $r\in R$ lies in only finitely many $P\in \Spec^1(R)$.
\end{enumerate}

If $R$ is a Krull domain, we denote the normalized discrete valuation
on the quotient field of $R$ whose valuation ring is $R_P$ by $\vP$.
Such a valuation is called an essential valuation of the Krull domain $R$.

Again, the existence of finite $D_P$-polynomially dense subsets of
$S$ relative to $\F$ in the following theorem is guaranteed by
Proposition \ref{finite_dense_set}.

\begin{thm}\label{Krull_divisor_hom}
Let $D$ be a Krull domain with quotient field $K$ and $S\subseteq D$
such that $S$ doesn't have any isolated points in $v$-adic topology
for any essential valuation $v$ of $D$. 
Let $f\in\Int(S,D)$, and $\monoid{f}$ the divisor-closed multiplicative
submonoid of $\Int(S,D)$ generated by $f$.

Let $H$ be the finite set of different monic 
irreducible factors of $f$ in $K[x]$ and $\H$ the multiplicative
submonoid of $K[x]$ generated by $H$ and the non-zero constants.

Let $\P$ be the finite set of primes $P$ of height $1$ of $D$ such 
that either $f\not\in D_P[x]$ or $f\in D_P[x]$ and $\vP(f(S))> 0$.
For each $P\in\P$, let $T_P$ be a finite subset of $S$ that is 
$D_P$-polynomially dense relative to $\H$ in $S$ and contains no
root of $f$.

Let 
\[(M,+)=
\sum_{h\in H}(\N_0,+)\oplus \sum_{P\in\P}\sum_{t\in T_P}(\N_0,+).
\]
Then
\[
\varphi\colon \monoid{f}\rightarrow M,\qquad
\varphi(g) = 
\left((v_h(g)\mid h\in H), ((\vP(g(t))\mid t\in T_P)\mid P\in\P)\right),
\]
is a divisor homomorphism.
\end{thm}

\begin{proof}
It is clear that $\varphi$ is a monoid homomorphism. Now assume
$a,b\in\monoid{f}$ with $\varphi(a)\divides\varphi(b)$. It suffices
to show that $a$ divides $b$ in $K[x]$ and that the co-factor 
$c=b/a$ is in $\Int(S,D_P)$ for all $P\in \Spec^1(D)$, because then
$c\in\Int(S,D)$, which implies that $a$ divides $b$ in $\monoid{f}$
by Remark \ref{divisibility}.

Let $c=b/a$. That $c$ is in $K[x]$ follows from $v_h(a)\le v_h(b)$
for all irreducible factors $h$ of $a$ and $b$ in $K[x]$.

Consider a prime $P$ of height $1$ of $D$ that is not
in $\P$. For such a prime, $f\in D_P[x]$ and $f$ is 
image-primitive in $\Int(S,D_P)$. We may apply Lemma 
\ref{imageprimitive_f}~(3) and deduce that $c\in\Int(S,D_P)$.

Now for $P\in \P$, let $\psi_P$ be the projection of $M$ onto 
$\sum_{h\in H}(\N_0,+)\oplus \sum_{t\in T_P}(\N_0,+)$, and call 
the latter monoid $M(P)$. From $\varphi(a)\mid\varphi(b)$
it follows that $\psi_P(\varphi(a))$ divides $\psi_P(\varphi(b))$.
Let $\monoid{f}_P$ be the divisor closed submonoid of $\Int(S,D_P)$
generated by $f$. Then $\monoid{f}$ is a submonoid of $\monoid{f}_P$,
and $\psi_P\circ\varphi$ is the restriction to $\monoid{f}$ of the
divisor homomorphism in \ref{local_divisor_hom}. Now the fact that
$\psi_P(\varphi(a))$ divides $\psi_P(\varphi(b))$ implies
$c\in\Int(S,D_P)$, by Proposition $\ref{local_divisor_hom}$.
\end{proof}

\begin{cor}\label{monadically_Krull}
Let $D$ be a Krull domain and $S$ a subset that doesn't have any 
isolated points in any of the topologies given by essential
valuations of $D$. Let $f\in \Int(S,D)$. Then $\monoid{f}$, the
divisor closed submonoid of $\Int(S,D)$ generated by $f$, is a
Krull monoid. 

In particular, for every Krull domain $D$ and
every $f\in\Int(D)$, the divisor closed submonoid $\monoid{f}$
of $\Int(D)$ generated by $f$ is a Krull monoid.
\end{cor}

\begin{proof}
Indeed, the existence of a divisor homomorphism from $\monoid{f}$ 
to a finite sum of copies of $(\N_0,+)$ in Theorem \ref{Krull_divisor_hom}
ensures that $\monoid{f}$ is a Krull monoid, see 
\cite{GeHK06nuf}[Thm.~2.4.8].
\end{proof}

Monoids with the property that the divisor closed submonoid 
generated by any single element is a Krull monoid have been called 
\textit{monadically Krull} by A.~Reinhart. Without using divisor
homomorphisms, through an approach completely different from ours,
Reinhart showed that $\Int(D)$ is monadically Krull whenever $D$ is 
a principal ideal domain \cite{Rei14monK}[Thm.~5.2]. 

Corollary \ref{monadically_Krull} generalizes Reinhart's result to Krull 
domains, and to integer-valued polynomials on sufficiently nice subsets.
The explicit divisor homomorphisms of Proposition~\ref{local_divisor_hom}
and Theorems~\ref{monadic_local_divisor_theory} and \ref{Krull_divisor_hom}
give additional information on the arithmetic of submonoids of $\Int(D)$. 
It remains an open problem to find the precise divisor theories 
(cf.~Def.~\ref{divisor_hom_def}) of those monoids of integer-valued
polynomials for which the above theorems provide divisor homomorphisms.
\bigskip

\textbf{Acknowledgment.} Enthusiastic thanks go out to the anonymous
referee for his/her meticulous reading of the paper and the resulting
corrections.
\bigskip

\bibliography{complete-dh}
\bibliographystyle{siamese}
\enlargethispage{2cm}
\end{document}